\documentclass{amsart}

\usepackage{latexsym,amsfonts}
\usepackage{amsmath,amssymb,graphics,setspace}
\usepackage{amsthm}
\usepackage{MnSymbol}
\usepackage{mathrsfs}
\usepackage{fontenc}%

\usepackage{marvosym}
\usepackage{paralist}
\usepackage{color}

\usepackage{graphicx}
\usepackage{fontenc}%

\usepackage{hyperref}
\hypersetup{linktocpage=true,colorlinks=true,linkcolor=blue,citecolor=blue,pdfstartview={XYZ 1000 1000 1}}

\usepackage[verbose]{wrapfig}

\usepackage{subfigure}
\renewcommand{\thesubfigure}{\thefigure.\arabic{subfigure}}
\makeatletter
\renewcommand{\p@subfigure}{}
\renewcommand{\@thesubfigure}{\thesubfigure:\hskip\subfiglabelskip}
\makeatother

\newtheorem{theorem}{Theorem}
\newtheorem{lemma}{Lemma}
\newtheorem{corollary}{Corollary}

\theoremstyle{definition}

\newcommand{\abs}[1]{\lvert#1\rvert}

\begin{document}

\newcommand*\discR{
 \psframe*[linecolor=red!80](0.28,0.28)
 }
\newcommand*\discG{
 \psframe*[linecolor=green!80](0.28,0.28)
 }
\newcommand*\discB{
 \psframe*[linecolor=blue!80](0.28,0.28)
 }
\newcommand*\discY{
 \psframe*[linecolor=yellow!90](0.28,0.28)
 }
\newcommand*\Grid{
 \psgrid[gridlabels=0,subgriddiv=3](0,0)(3,3)
 }
\newcommand*\Gridtwo{
 \psgrid[gridlabels=0,subgriddiv=3](0,0)(9,9)
 }

%
%
%
%
%

\title{A Note on Computable Proximity of $\mathcal{L}_1$-Discs\\ on the Digital Plane}

\author[J.F. Peters]{J.F. Peters$^{\alpha}$}
\email{James.Peters3@umanitoba.ca}
\address{\llap{$^{\alpha}$\,}Computational Intelligence Laboratory,
University of Manitoba, WPG, MB, R3T 5V6, Canada and
Department of Mathematics, Faculty of Arts and Sciences, Ad\.{i}yaman University, 02040 Ad\.{i}yaman, Turkey}

\author[K. Kordzaya]{K. Kordzaya$^{\beta}$}
\email{korka@ciu.edu.ge}

\author[I. Dochviri]{I. Dochviri$^{\beta}$}
\email{iraklidoch@yahoo.com}

\address{\llap{$^{\beta}$\,}Department of Mathematics, Caucasus International University,
73, Chargali str., 0192 Tbilisi, Georgia}


\date{}

%

%


\begin{abstract}
{This paper investigates problems in the characterization of the proximity of digital discs.  Based on the $\mathcal{L}_1$-metric structure for the 2D digital plane and using a Jaccard-like metric, we determine numerical characters for intersecting digital discs.}
\end{abstract}



%

\maketitle

\section{Introduction}
This paper introduces a form of digital geometry in proximity spaces.  The study of digital discs is connected to the discovery of proximal objects~\cite{Peters2016ISRLcomputationalProximity,DBLP:series/isrl/2014-63,Irakli2016MCStopologicalSorting}.  The objects often can be represented as sets of points and this stipulates that set-theoretic and topological methods are very useful tools in the study of proximity relations.  Digital geometry deals with geometric properties of objects on computer screens~\cite{Klette2004,Kopperman1991AMMonthyDigitalTopology,Kronheimer1992,Rosenfeld1979}.

\setlength{\intextsep}{0pt}
\begin{wrapfigure}[8]{R}{0.25\textwidth}
\begin{minipage}{3.2 cm}
\centering
\includegraphics[width=25mm]{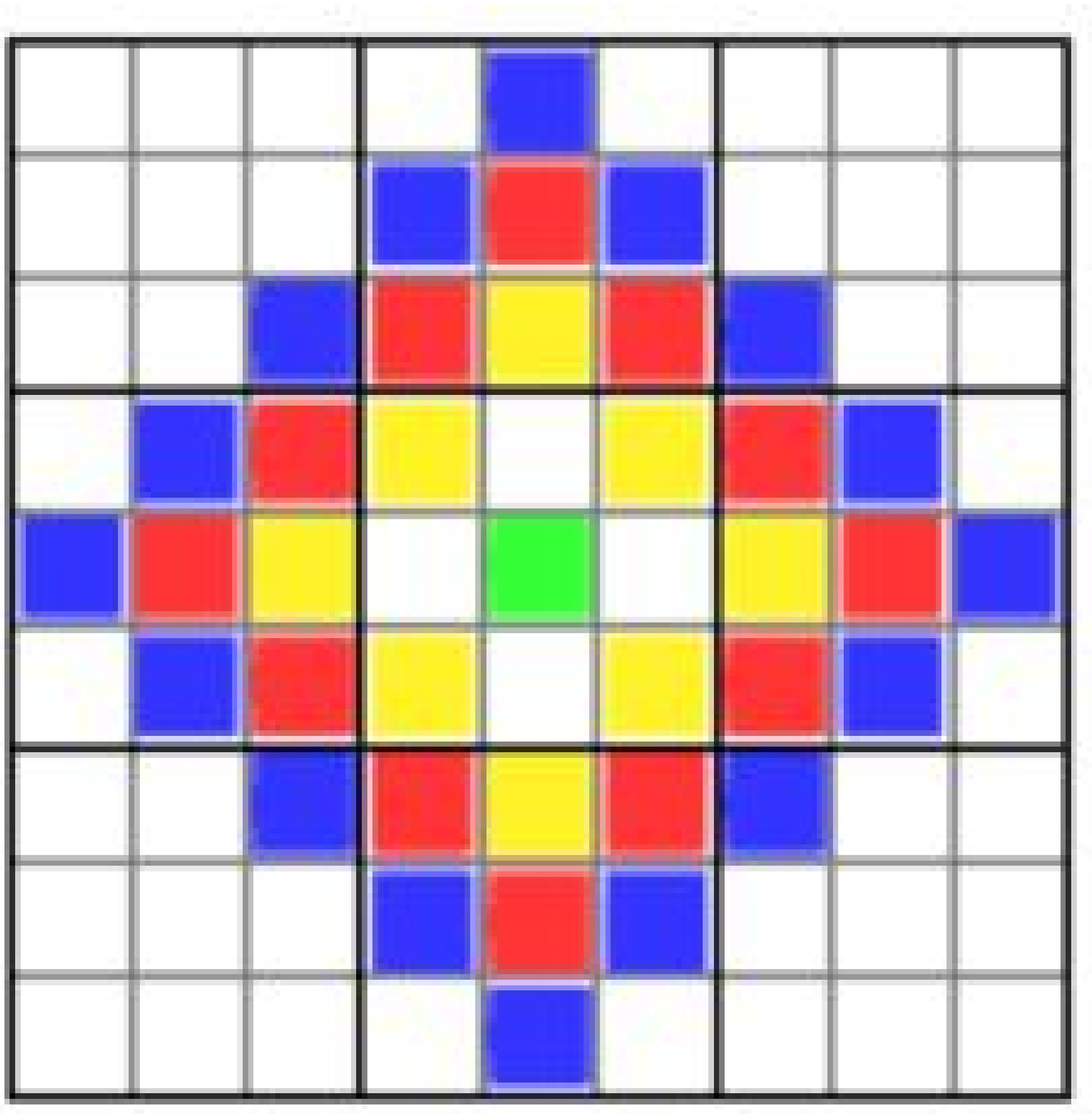}
\caption[]{\footnotesize Structures}
\label{fig:Digital discs}
\end{minipage}
\end{wrapfigure}

Many different computer screen images can be obtained via pixel lighting.  A \emph{pixel} is the smallest element is a digital image and are usually identified as points.   In other words, we can describe images on the computer screen by their pixels that have digital valued coordinates, {\em i.e.}, a mathematical model of the computer screen is the digital plane $\mathbb{Z}^2$.

The importance of the notions of the circle and disc in Euclidean geometry is well known.   In digital geometry, digital circles and digital discs have various important properties that are different from the Euclidean ones (see, {\em e.g.},~\cite{Nakamura1984CVGIPdigitalCircles,McIlroy1983ACMTGintegerGrids,Kim1984PAMIdigitalDiscs,Toutant2013DAMdigitalCircles,Andres2011LNCSdigitalCircles}).  One of the reasonable realizations of metric structure on the digital plane $\mathbb{Z}^2$ can be determined via the so-called $\mathcal{L}_1$ metric.  This metric has the following view:
\[
d\left(p_1,p_2\right) = \abs{a_1-a_2} + \abs{b_1-b_2}, \mbox{where $p_1$ and $p_2$ are some matched points},
\]
{\em i.e.}, $p_1$ and $p_2$ are pixels for our future considerations.   Since we can represent pixel coordinates as digital pairs, then it is obvious that $d\left(p_1,p_2\right)\in \mathbb{Z}$ (the integers).

Based on the $\mathcal{L}_1$ metric, we define a digital circle
with radius $r$ and center $x$ (denoted by $C_d(x,r)$) as follows:
\[
C_d(x,r) = \left\{z\in \mathbb{Z}^2: d(x,z) = r\right\}.
\]
Moreover, we denote by $c\left(C_d(x,r)\right)$ the circumference
of the circle $C_d(x,r)$ where $r\in
\mathbb{N}\cup\left\{0\right\}$.

Due to R. Klette and A. Rosenfeld~\cite{Klette2004}, it is known
that $\pi_{\mathcal{L}_1} =
\frac{c\left(C_d(x,r)\right)}{\mbox{diam}\left(C_d(x,r)\right)} =
\frac{8r}{2r} = 4$, where $\mbox{diam}\left(C_d(x,r)\right)$ is
the diameter of the circle $C_d(x,r)$.  Using this fact, we easily
obtain the following result.

\begin{lemma}\label{lemma:digitalCircle}
Let $C_d(x,r)$ be a digital circle with center at point $x$ and radius $r$ relative to the $\mathcal{L}_1$ metric.
Then, for the number of pixels of $C_d(x,r)$, we have the formula
\[
\mbox{card}\left(C_d(x,r)\right) = \frac{2c\left(C_d(x,r)\right)}{\pi_{\mathcal{L}_1}} = 4r.
\]
\end{lemma}

\noindent Fig.~\ref{fig:Digital discs} demonstrates the structural property of the digital disc, namely,
\begin{align*}
D_d\left(x,R\right) &= \left\{z\in\mathbb{Z}^2\mid d(x,z)\leq R\right\},\ \mbox{particularly:}\\
D_d\left(x,R\right) &=
\left\{x\right\}\cup\left(\mathop{\bigcup}\limits_{r=1}^R
C_d(x,r)\right), \ \mbox{where}\ R\in\mathbb{Z}.
\end{align*}

\begin{lemma}\label{lemma:discPixels}
If $D_d\left(x,R\right)$ is a digital disc relative to the
$\mathcal{L}_1$ metric $d$, then the number of pixels forming the
disc $D_d\left(x,R\right)$ can be computed by the formula
$\mbox{card}\left(D_d\left(x,R\right)\right) = 2R^2 + 2R + 1$.
\end{lemma}
\begin{proof}
Since $D_d\left(x,R\right) = \left\{x\right\}\cup\left(\mathop{\bigcup}\limits_{r=1}^R C_d(x,r)\right)$, we can write
\[
\mbox{card}\left(D_d\left(x,R\right)\right) = 1 + \mbox{card}\left(C_d(x,1)\right) + \mbox{card}\left(C_d(x,2)\right) +\cdots+\mbox{card}\left(C_d(x,R)\right).
\]
Now, applying Lemma~\ref{lemma:digitalCircle}, we get
\begin{align*}
\mbox{card}\left(D_d\left(x,R\right)\right) &= 1 + 4 + 8 +\cdots+4R=\\
   &= 1 + 4\left(\frac{1 + R}{2}R\right)=\\
     &= 2R^2 + 2R + 1.
\end{align*}
\end{proof}

\section{How Near are Digital Discs?}
To solve a wide class of the problems of computational proximity,
we know that the Hausdorff metric is
appropriate~\cite{Klette2004,Deza2009}.  The Hausdroff metric
(denoted by $d_H(A,B)$) measures the distance between the sets
$A,B$ in the given metric space $(X,d)$ and is defined by
\[
d_H(A,B) = \mbox{max}\left\{\mathop{\mbox{sup}}\limits_{x\in A}\mathop{\mbox{inf}}\limits_{y\in B}d(x,y),
\mathop{\mbox{sup}}\limits_{y\in B}\mathop{\mbox{inf}}\limits_{x\in A}d(x,y)\right\}.
\]
If the sets $A,B$ are finite, we obtain the simplication of the Hausdorff metric by maxima and minima~\cite{Engelking1989}, {\em i.e.},
\[
d_H(A,B) = \mbox{max}\left\{\mathop{\mbox{max}}\limits_{x\in A}\mathop{\mbox{min}}\limits_{y\in B}d(x,y),
\mathop{\mbox{max}}\limits_{y\in B}\mathop{\mbox{min}}\limits_{x\in A}d(x,y)\right\}.
\]
For intersecting sets $A$ and $B$, \emph{i.e.}, $A\cap B\neq
\emptyset$, the Hausdorff metric guarantees that $d_H(A,B) = 0$.
Such sets in the theory of proximity spaces~\cite[\S
8.4]{Engelking1989} are said to be trivially near.  Therefore, if
$A\cap B\neq \emptyset$ and $A\cap C\neq \emptyset$ hold in the
metric space $(X,d)$, we cannot distinguish which the sets in the
pair $B,C$ is more near to $A$.  Hence, the application of
Hausdorff distance in the sorting of near sets is more suitable
for disjoint sets.

\setlength{\intextsep}{0pt}
\begin{wrapfigure}[10]{R}{0.25\textwidth}
\begin{minipage}{3.2 cm}
\centering
\includegraphics[width=35mm]{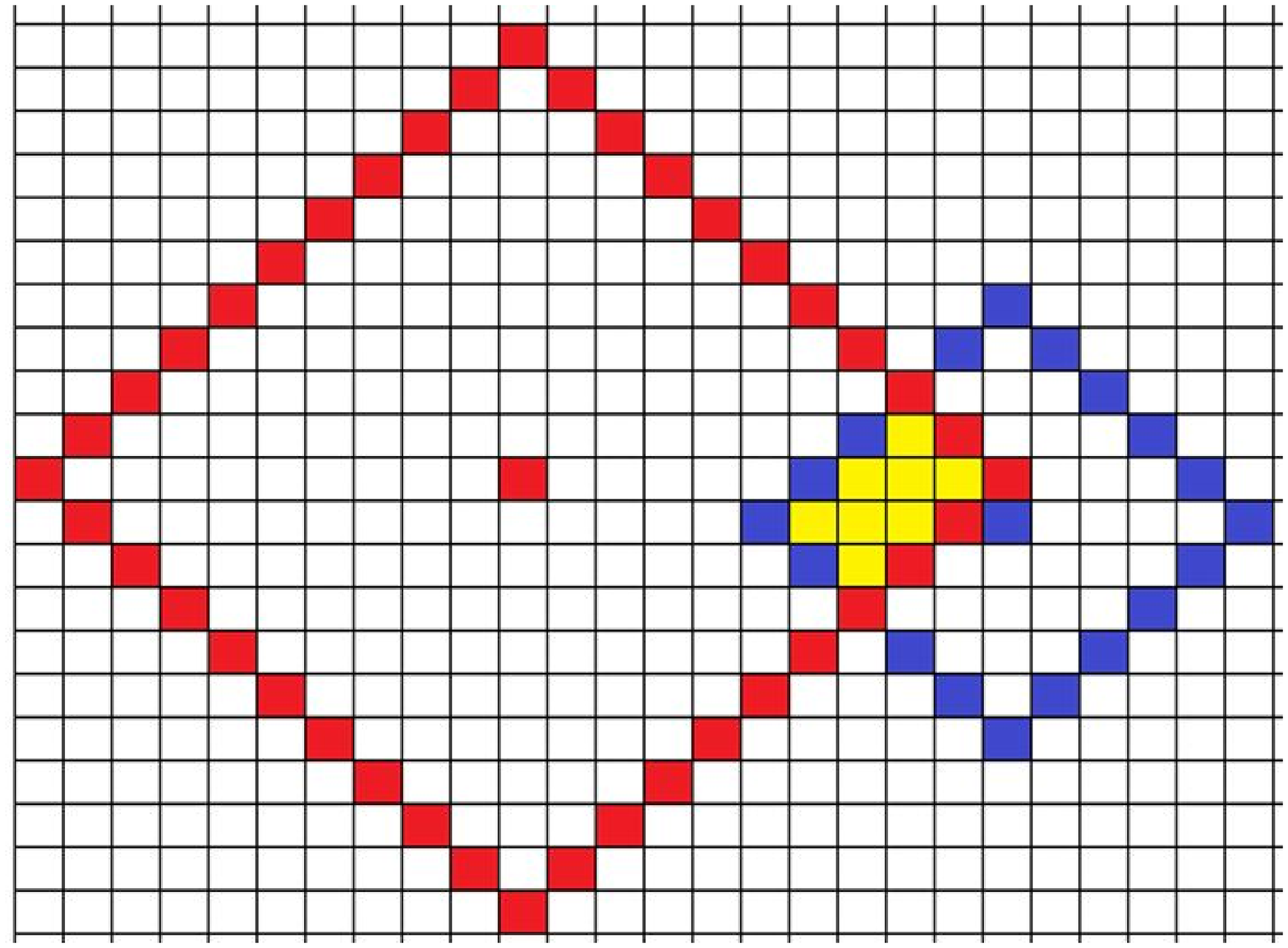}
\caption[]{Overlap}
\label{fig:intersectionDiscs}
\end{minipage}
\end{wrapfigure}
$\mbox{}$\\
\vspace{2mm}

Classification of images in computer science frequently need the application of Jaccard-like metrics~\cite{Fujita2013JJIAMsetDistance}.   We will use a simplified version to analyze proximity of intersecting digital discs.  It must be especially noticed that the problem connected with the intersection of plane discs was considered from a computer science perspective in~\cite{Sharir1985SIAMJCplanarDiscs}.

For the Jaccard-like metric $m$, we understand the distance
function defined via the cardinality of the symmetric difference
of two arbitrary nonempty finite sets $A$ and $B$, {\em i.e.},
\begin{align*}
m(A,B) &= \mbox{card}\left(A\bigtriangleup B\right)\\
       &= \mbox{card}\left(A\setminus B\right) + \mbox{card}\left(B\setminus A\right)\\
             &= \mbox{card}\left(A\right) + \mbox{card}\left(B\right) - 2\mbox{card}\left(A\cap B\right).
\end{align*}

It is obvious that if $\mbox{card}\left(A\right)\neq \mbox{card}\left(B\right)$ and both sets are finite while $A\cap B\neq\emptyset$, we get $m(A,B) \neq 0$.  This raises the question of the computation of the proximity of intersecting digital discs such as the ones in Fig.~\ref{fig:intersectionDiscs}.

\begin{theorem}
Let $D_d(x,R_1)$ and $D_d(y,R_2)$ be digital discs such that
$C_d(x,R_1)\cap C_d(y,R_2)\neq\emptyset$.  Then
\[
m\left(D_d(x,R_1),D_d(y,R_2)\right) = 2\left(R_1^2 + R_2^2 + R_1 +
R_2 - 2kn + k + n\right),
\]
where $k$ and $n$ denote the number of pixels forming the width and height of the greatest rectangle subset of an intersection set.
\end{theorem}
\begin{proof}
Appling Lemma~\ref{lemma:discPixels}, we obtain the following
cardinal equalities:
\begin{align*}
m\left(D_d(x,R_1),D_d(y,R_2)\right) &= \mbox{card}\left(D_d(x,R_1)\right) + \mbox{card}\left(D_d(y,R_2)\right) -
2\mbox{card}\left(D_d(x,R_1)\cap D_d(y,R_2)\right)\\
  &= 2\left(R_1^2 + R_2^2 + R_1 + R_2 + 1)\right) - 2\mbox{card}\left(D_d(x,R_1)\cap D_d(y,R_2)\right)\\
    &= 2\left(R_1^2 + R_2^2 + R_1 + R_2 + 1\right) - 2\left[kn + (k-1)(n-1)\right]\\
    &= 2\left(R_1^2 + R_2^2 + R_1 + R_2 - 2kn + k + n\right)
\end{align*}
\end{proof}

\setlength{\intextsep}{0pt}
\begin{wrapfigure}[11]{R}{0.25\textwidth}
\begin{minipage}{3.2 cm}
\centering
\includegraphics[width=35mm]{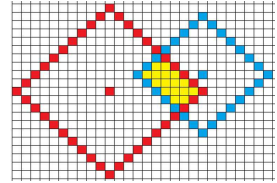}
\caption[]{Non-Intersecting Boundaries}
\label{fig:discBdys}
\end{minipage}
\end{wrapfigure}
$\mbox{}$\\
\vspace{2mm}

Notice that there is a situation in which two digital discs are intersecting but their boundaries are not intersecting (see, {\em e.g.},~Fig.\ref{fig:discBdys}).   Observe that in that case, we have $C_d\left(x,R_1-1\right)\cap C_d\left(y,R_2\right)\neq\emptyset$, or, equivalently, $C_d\left(x,R_1\right)\cap C_d\left(y,R_2-1\right)\neq\emptyset$.

\begin{theorem}
Let $D_d\left(x,R_1\right)$ and $D_d\left(y,R_2\right)$ be digital discs such that $C_d\left(x,R_1\right)\cap C_d\left(y,R_2\right) = \emptyset$, but $C_d\left(x,R_1-1\right)\cap C_d\left(y,R_2\right)\neq\emptyset$.  Then we have
$\mbox{m}\left(D_d\left(x,R_1\right),D_d\left(y,R_2\right)\right) = 2\left(R_1^2+R_2^2+R_1+R_2+1-2kn\right)$, where $k$ and $n$ denote the number of pixels forming the width and height of the greatest rectangle subset of an intersection set.
\end{theorem}
\begin{proof}
In this case, we can easily not that $\mbox{card}\left(D_d\left(x,R_1\right)\cap D_d\left(y,R_2\right)\right) = 2kn$.  Hence, we have $\mbox{m}\left(D_d\left(x,R_1\right),D_d\left(y,R_2\right)\right) = 2\left(R_1^2+R_2^2+R_1+R_2+1-2kn\right)$.
\end{proof}

Next, we need to represent the centers $x$ and $y$ of discs $D_d\left(x,R_1\right)$ and $D_d\left(y,R_2\right)$ by a couple of digital coordinates as follows: $x = \left(\alpha,\beta\right)$ and $y = \left(\gamma,\delta\right)$.  If one of the following equalities hold $d(x,y) = \abs{\alpha - \gamma}$ or $d(x,y) = \abs{\beta - \delta}$, {\em i.e.}, the centers of the discs lie on horizontal or verical axes (similar to the situations shown in Fig.~\ref{fig:discCentres} and Fig.~\ref{fig:nonCapBdys}), then we can measure the proximity of the discs via computation of the pixel cardinality of the intersections sets.

\begin{figure}[!ht]
\centering
\subfigure[Intersecting Discs with Intersecting Boundaries]
 {\label{fig:discCentres}\includegraphics[width=50mm]{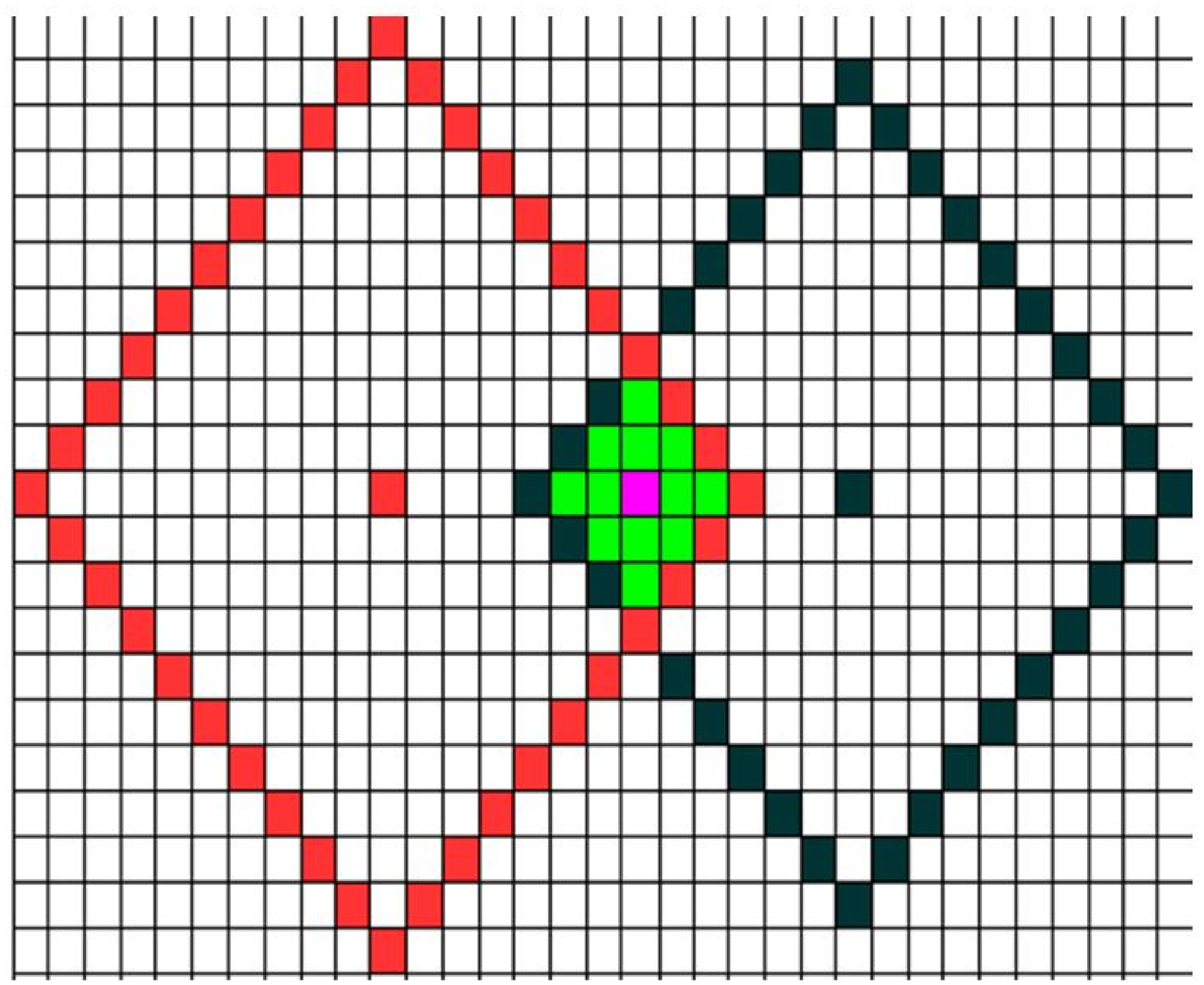}}\hfil
\subfigure[Intersecting Discs with Non-Intersecting Boundaries]
 {\label{fig:nonCapBdys}\includegraphics[width=50mm]{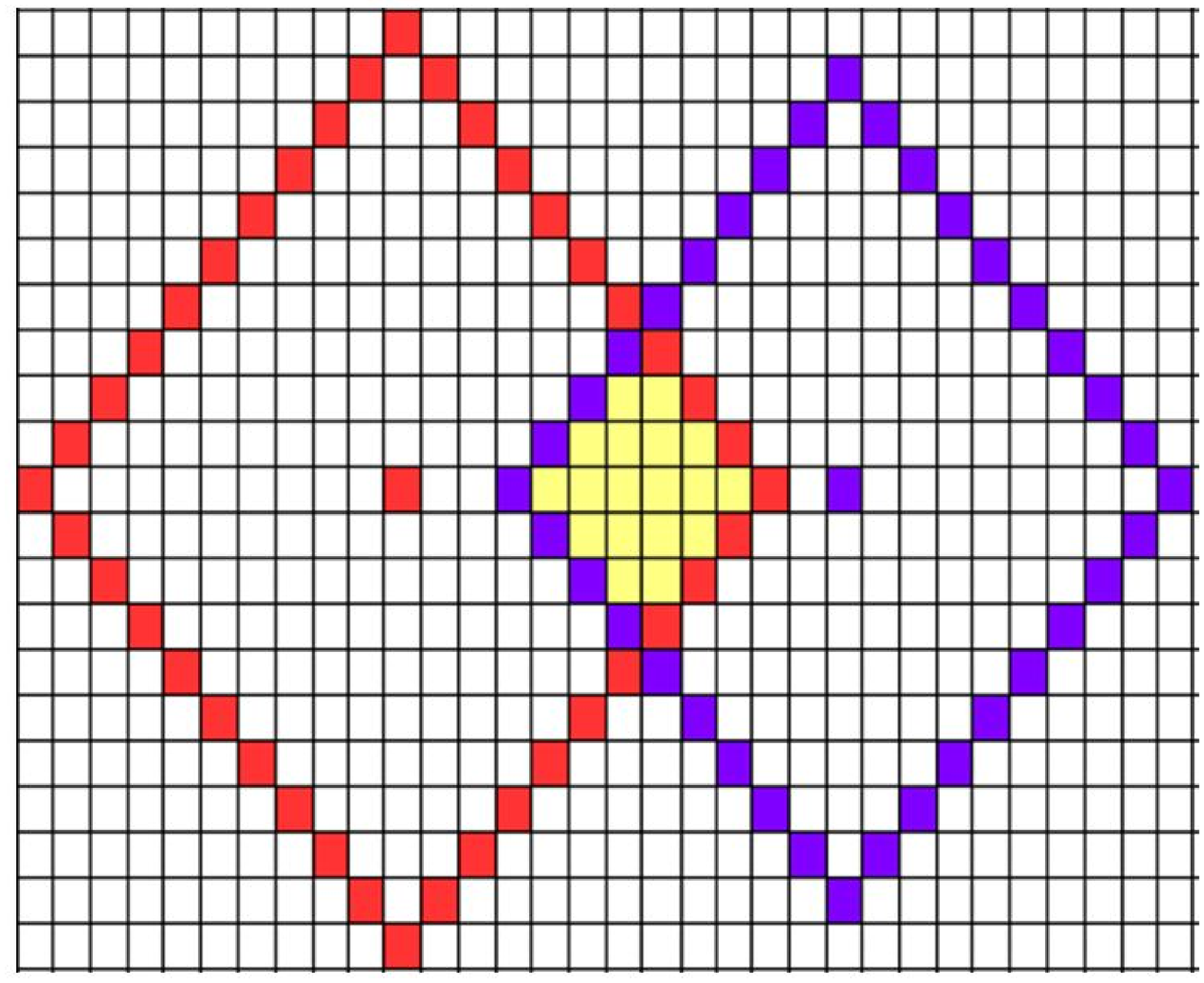}}\hfil
\caption[]{Intersecting Discs on the Digital Plane}
\label{fig:intersectingDiscs}
\end{figure}
$\mbox{}$\\
\vspace{2mm}

\begin{theorem}\label{thm:coords}
Let $D_d\left(x,R_1\right)$ and $D_d\left(y,R_2\right)$ be digital discs such that $x = \left(\alpha,0\right)$ and $y = \left(\gamma,0\right)$ with $\alpha < \gamma$ and $\gamma - \alpha \leq R_1 + R_2$.  If $C_d\left(x,R_1\right)\cap C_d\left(y,R_2\right)\neq\emptyset$, then
\[
\mbox{m}\left(D_d\left(x,R_1\right),D_d\left(y,R_2\right)\right) = \left(R_1 - R_2\right)^2 + 2\left(R_1 + R_2 + 1\right)\left(\gamma - \alpha\right) - \left(\gamma - \alpha\right)^2.
\]
\end{theorem}
\begin{proof}
Since $x = \left(\alpha,0\right)$, $y = \left(\gamma,0\right)$ and $C_d\left(x,R_1\right)\cap C_d\left(y,R_2\right)\neq\emptyset$, we claim that
\[
C_d\left(x,R_1\right)\cap C_d\left(y,R_2\right) = C_d(k,r),\ \mbox{where},
\]
$k = \left(\frac{\alpha + R_1 + \gamma - R_2}{2},0\right)$ and\\
$r = R_1 - (k - \alpha) = \frac{R_1 + R_2 + (\gamma - \alpha)}{2}\in \mathbb{N}\cup \left\{0\right\}$.
Consequently, simplification of 
\[
\mbox{m}\left(D_d\left(x,R_1\right),D_d\left(y,R_2\right)\right) = 2\left(R_1^2 + R_2^2 + R_1 + R_2 + 1 - 2r^2 - 2r - 1\right)
\]
gives the needed expression
\[
\mbox{m}\left(D_d\left(x,R_1\right),D_d\left(y,R_2\right)\right) = \left(R_1 - R_2\right)^2 + 2\left(R_1 + R_2 + 1\right)\left(\gamma - \alpha\right) - \left(\gamma - \alpha\right)^2.
\]
\end{proof}

Observe that Theorem~\ref{thm:coords} can be applied in similar cases when the intersection set of the digital discs itself is a disc.

This leads us to consider two intersecting digital discs with non-intersecting boundaries (see, {\em e.g.}, Fig.~\ref{fig:nonCapBdys}) so that both centers lie on the horizontal or vertical axes.   In such cases, we obtain the following result.

\begin{corollary}
Let $D_d\left(x,R_1\right)$ and $D_d\left(y,R_2\right)$ be intersecting digital discs that satisfy the conditions of Theorem~\ref{thm:coords}, but $C_d\left(x,R_1\right)\cap C_d\left(y,R_2\right) = \emptyset$.   Then we have
\begin{align*}
\mbox{m}\left(D_d\left(x,R_1\right),D_d\left(y,R_2\right)\right) &= 2\left(R_1^2 + R_2^2 + R_1 + R_2 - 2r_0^2 - 4r_0 + 1\right), \mbox{where},\\
r_0 &= \frac{R_1 - 1 + R_2 + (\gamma - \alpha)}{2}.
\end{align*}
\end{corollary}

\section*{Acknowledgements}
J.F. Peters was supported by the Scientific and Technological Research
Council of Turkey (T\"{U}B\.{I}TAK) Scientific Human Resources
Development (BIDEB) under grant no: 2221-1059B211402463 and the
Natural Sciences \& Engineering Research Council of Canada (NSERC)
discovery grant 185986. I. Dochviri was supported by Shota
Rustaveli Georgian NSF Grant \ FR/291/5-103/14.

\bibliographystyle{amsplain}
\bibliography{NSrefs}

\providecommand{\bysame}{\leavevmode\hbox to3em{\hrulefill}\thinspace}
\providecommand{\MR}{\relax\ifhmode\unskip\space\fi MR }
\providecommand{\MRhref}[2]{%
  \href{http://www.ams.org/mathscinet-getitem?mr=#1}{#2}
}
\providecommand{\href}[2]{#2}
\begin{thebibliography}{10}

\bibitem{Andres2011LNCSdigitalCircles}
E.~Andres and T.~Roussillon, \emph{Analytical description of digital circles},
  Lecture Notes in Comput. Sci. \textbf{6607} (2011), 901--917, MR2833897.

\bibitem{Deza2009}
E.~Deza and M.-M. Deza, \emph{Encyclopedia of distances}, Springer, Berlin,
  2009.

\bibitem{Irakli2016MCStopologicalSorting}
I.~Dochviri and J.F. Peters, \emph{Topological sorting of finitely near sets},
  Mathematics in Computer Science (2016), 1--5, DOI: 10.1007/s11786-016-0273-1.

\bibitem{Engelking1989}
R.~Engelking, \emph{General topology, revised \& completed edition}, Heldermann
  Verlag, Berlin, 1989.

\bibitem{Fujita2013JJIAMsetDistance}
O.~Fujita, \emph{Metrics based on average distance between sets}, Jpn. J. Ind.
  Appl. Math. \textbf{30} (2013), no.~1, 1--19, MR3022803.

\bibitem{Kim1984PAMIdigitalDiscs}
C.E. Kim, \emph{Digital discs}, IEEE Transactions on Pattern Analysis and
  Machine Intelligence \textbf{6} (1984), no.~3, 372--374.

\bibitem{Klette2004}
R.~Klette and A.~Rosenfeld, \emph{Digital geometry. geometric methods for
  digital picture analysis}, Morgan-Kaufmann Pub., Amsterdam, The Netherlands,
  2004.

\bibitem{Kopperman1991AMMonthyDigitalTopology}
R.~Kopperman, T.Y. Kong, and P.R. Meyer, \emph{A topological approach to
  digital topology}, The American Math. Monthly \textbf{98} (1991), no.~10,
  901--917, MR1137537.

\bibitem{Kronheimer1992}
E.H. Kronheimer, \emph{The topology of digital images. {S}pecial issue on
  digital topology}, Topology and its Applications \textbf{46} (1992), no.~3,
  279--303, MR1198735.

\bibitem{McIlroy1983ACMTGintegerGrids}
M.D. McIlroy, \emph{Best approximate circles on integer grids}, ACM
  Transactions on Graphics \textbf{2} (1983), no.~4, 237--263.

\bibitem{Nakamura1984CVGIPdigitalCircles}
A.~Nakamura and K.~Aizawa, \emph{Digital circles}, Computer vision, graphics,
  and image processing \textbf{26} (1984), no.~2, 242--255.

\bibitem{DBLP:series/isrl/2014-63}
J.F. Peters, \emph{Topology of digital images - visual pattern discovery in
  proximity spaces}, Intelligent Systems Reference Library, vol.~63, Springer,
  2014, xv + 411pp, Zentralblatt MATH Zbl 1295 68010.

\bibitem{Peters2016ISRLcomputationalProximity}
\bysame, \emph{Computational proximity. {E}xcursions in the topology of digital
  images}, Springer Int. Pub., Intelligent Systems Reference Library, vol. 102,
  Switzerland, 2016, xxiii+433 pp., ISBN: 978-3-319-30262-1, doi:
  10.1007/978-3-319-30262-1.

\bibitem{Rosenfeld1979}
A.~Rosenfeld, \emph{Digital topology}, The Amer. Math. Monthly \textbf{86}
  (1979), no.~8, 621--630, Amer. Math. Soc. MR0546174.

\bibitem{Sharir1985SIAMJCplanarDiscs}
M.~Sharir, \emph{Intersection and closest-pair problems for a set of planar
  discs}, SIAM J. Comput. \textbf{14} (1985), no.~2, 448--468, MR0784749.

\bibitem{Toutant2013DAMdigitalCircles}
J.-L. Toutant, E.~Andres, and T.~Roussillon, \emph{Digital circles, spheres and
  hyperspheres: from morphological models to analytical characterizations and
  topological properties}, Discrete Appl. Math. \textbf{161} (2011), no.~16-17,
  2662--2677, MR3101744.

\end{thebibliography}

\end{document}